\documentclass[11pt,reqno]{amsart}

\usepackage[T1]{fontenc}
\usepackage{lmodern}
\usepackage[a4paper,margin=27mm]{geometry}
\usepackage{amsmath,amssymb,mathtools,microtype}

\usepackage[
  colorlinks=true,
  linkcolor=blue,
  citecolor=blue,
  urlcolor=blue
]{hyperref}

\hypersetup{
  pdftitle={Malliavin Smoothness of Hermite Processes of Arbitrary Order and Their Wiener Integrals},
  pdfauthor={Elina Moldavskaya},
  pdfsubject={Research manuscript, version 6, September 10, 2026}
}

\numberwithin{equation}{section}

\newtheorem{theorem}{Theorem}[section]
\newtheorem{proposition}[theorem]{Proposition}
\newtheorem{lemma}[theorem]{Lemma}
\newtheorem{corollary}[theorem]{Corollary}

\theoremstyle{definition}
\newtheorem{definition}[theorem]{Definition}

\theoremstyle{remark}
\newtheorem{remark}[theorem]{Remark}

\newcommand{\E}{\mathbb E}
\newcommand{\Prob}{\mathbb P}
\newcommand{\R}{\mathbb R}
\newcommand{\N}{\mathbb N}
\newcommand{\Hh}{\mathfrak H}
\newcommand{\HT}{\mathfrak H_T}
\newcommand{\Dt}{D_T}
\newcommand{\cK}{\mathcal K}
\newcommand{\Sch}{\mathcal S}
\newcommand{\Sphere}{\mathbb S}
\newcommand{\dd}{\,\mathrm d}
\newcommand{\norm}[1]{\lVert #1\rVert}
\newcommand{\ip}[2]{\langle #1,#2\rangle}
\newcommand{\ind}{\mathbf 1}

\DeclareMathOperator{\Ran}{Ran}
\DeclareMathOperator{\rank}{rank}

\DeclareMathOperator{\adj}{adj}

\title[Malliavin Smoothness of Hermite Processes]
{Malliavin Smoothness of Hermite Processes\\
of Arbitrary Order and Their Wiener Integrals}

\author{Elina Moldavskaya}

\thanks{Technion---Israel Institute of Technology, Israel.
E-mail: \href{mailto:elina.mol@technion.ac.il}{elina.mol@technion.ac.il}.}

\date{}



\begin{document}
\begin{abstract}
We establish Malliavin nondegeneracy for every weighted integral against a singular Hermite kernel, with all negative moments of the Malliavin derivative norm finite, uniformly over admissible compact families of bounded deterministic weights. In particular, all finite-dimensional distributions of Hermite processes of every fixed finite order have Schwartz densities, as do vectors of non-overlapping increments and, more generally, finite families of Wiener integrals with respect to a Hermite process. For such vectors, linear independence of the weights is both necessary and sufficient for absolute continuity; under this condition, all inverse Malliavin determinant moments exist and the joint density belongs to the Schwartz space, with bounds that are again uniform over compact families. The proof rests on a nonvanishing condition on compact weight families that is preserved under directional differentiation. An analytic-tail property of fractional transforms of the directions verifies this condition at every chaos level. Uniform Malliavin estimates then convert lower-order gradient bounds into bounded joint densities of arbitrarily many first directional derivatives, and Bessel's inequality closes the induction. The argument works directly with fixed non-Gaussian laws and singular kernels, without a Gaussian-limit assumption and without a separate local-nondeterminism transfer.
\end{abstract}
\maketitle

\section{Introduction}\label{sec:introduction}

Hermite processes are canonical limits of nonlinear functionals of long-range dependent Gaussian sequences in the non-central limit theorems of Dobrushin and Major~\cite{DobrushinMajor1979} and Taqqu~\cite{Taqqu1979}. They form a family of self-similar processes with stationary increments, indexed by a self-similarity parameter $H\in(1/2,1)$ and a positive integer $q$. The first order gives fractional Brownian motion, the second gives the Rosenblatt process, and the higher orders provide non-Gaussian limits associated with larger Hermite ranks. Regularity of these limiting laws is an important part of their distributional theory: bounded densities give anti-concentration estimates, while smoothness and decay permit differentiation and Fourier analysis of the laws. The multiple-integral representation alone does not supply these properties.

For a random vector $X=(X_1,\ldots,X_n)$ with Malliavin differentiable components, its Malliavin matrix is
\[
 \Gamma_X=\bigl(\ip{DX_i}{DX_j}_{\Hh}\bigr)_{1\le i,j\le n}.
\]
Almost sure invertibility of this matrix is sufficient for absolute continuity. To obtain smooth densities by the usual nondegeneracy criterion, one proves the stronger property
\begin{equation}\label{eq:intro-inverse}
 \E\bigl[(\det\Gamma_X)^{-p}\bigr]<\infty
 \qquad\text{for every }p>0.
\end{equation}
When $X\in(\mathbb D^\infty)^n$, this condition yields a density in the Schwartz space; see Nualart~\cite[Section 2.1]{Nualart2006}. The distinction between positivity and arbitrary inverse moments is essential. Nonconstancy, even within a fixed Wiener chaos, is insufficient: for a unit vector $e$, the centered quadratic variable $W(e)^2-1$ already fails to have inverse gradient moments of order one; see Remark~\ref{rem:kernelstructure}. Thus a smoothness theorem requires a structural nondegeneracy mechanism beyond the polynomial degree.

Loosveldt, Nachit, Nourdin and Tudor~\cite{LNNTAC} proved absolute continuity of the finite-dimensional distributions of Hermite processes of arbitrary order. Their method combines a Gram-factorization identity for the Malliavin determinant with strong local nondeterminism at the level of Malliavin derivatives. In Remark~7.1 of that work, they identify inverse moments of the determinant as the additional issue required for smoothness. Their subsequent Rosenblatt paper~\cite[p.~2]{LNNTRosenblatt} explicitly describes the difficulty of establishing these moments for general orders and obtains them for $q=2$, together with Schwartz densities and exponential-type bounds for density derivatives. At order three, Moldavskaya~\cite{Moldavskaya2026} established nondegeneracy through infinite-rank quadratic Gaussian forms arising from directional derivatives. That work also proves grid-uniform Sobolev bounds for inverse Malliavin determinants of normalized increment vectors and stretched-exponential density-derivative estimates with exponent $2/3$.

The main contribution of the present paper is a uniform nondegeneracy theorem for singular weighted Hermite kernels. For the integrals $G_r(w)$ defined in Section~\ref{sec:main}, Theorem~\ref{thm:mainuniform} proves
\[
 \sup_{w\in\cK}\E\bigl[\norm{D_TG_r(w)}_{L^2(0,T)}^{-p}\bigr]<\infty,
 \qquad p>0,
\]
for every norm-compact family $\cK$ of nonzero real bounded weights, with the order and kernel parameters fixed. Equivalently, the gradient has small-ball estimates of every polynomial order, uniformly over the whole family. Theorem~\ref{thm:mainvector} then gives an exact criterion for weighted vectors: linear independence of their deterministic weights is necessary and sufficient for absolute continuity, and sufficiency yields all inverse Malliavin determinant moments and a Schwartz density. This criterion depends only on the weights and the kernel class; it does not assume independence of the random coordinates. General criteria for vectors in a fixed chaos, such as those of Nourdin, Nualart and Poly~\cite{NourdinNualartPoly2013}, decide absolute continuity but supply no inverse moments, and therefore no smoothness.

The structural ingredient is an induction hypothesis that remains valid after every directional differentiation. The identity
\[
 D_hG_r(w)=rG_{r-1}(wa_h)
\]
reduces the chaos order while multiplying the weight by a fractional transform $a_h$ of the direction. If $h$ is supported before a fixed time $\tau$, this transform is analytic and nonzero almost everywhere after $\tau$. Consequently, multiplication preserves nonvanishing on a set of positive measure in the same terminal interval. Lemmas~\ref{lem:tail} and~\ref{lem:stable} show that this property survives uniformly over spheres of directions and compact weight families. Lemma~\ref{lem:commontail} supplies a common terminal interval for every norm-compact family excluding zero. The resulting invariant class is isolated as the admissible compact families of Definition~\ref{def:admissible}, and it is this family-level assertion, rather than a statement about one unweighted law, that forms a closed induction hypothesis.

At each step, the lower-order estimates give inverse Malliavin eigenvalue moments for an arbitrarily large vector of first directional derivatives. Uniform integration by parts produces a bounded joint density, and Bessel's inequality turns this into arbitrary-order small-ball estimates for the original gradient. The induction begins with the injectivity of a right-sided fractional integral at the first chaos level. It works directly with the limiting kernels and with fixed non-Gaussian laws; neither an explicit higher-chaos characteristic function nor a Gaussian approximation is required.

\medskip
\noindent\emph{Relation to general regularity methods.}
The use of differentiation to reduce polynomial degree has substantial antecedents. In Theorem~3.1 of Bogachev~\cite{Bogachev1992}, integral polynomials of random fields have Schwartz densities under a structural condition on admissible differentiation directions. Its proof uses lower-degree derivatives and their joint distributions, together with control of parameter dependence. This is a relevant predecessor of the general proof scheme. Broader stratification and differential-operator methods for stochastic functionals are developed by Davydov, Lifshits and Smorodina~\cite{DLS1998}; the Gaussian-measure framework, including Sobolev classes and nonlinear transformations, is treated in Bogachev~\cite{Bogachev1998}. The singular kernels considered here do not admit the immediate pointwise-field substitution underlying the integral-polynomial representation: $(s-\cdot)_+^{-\beta}\notin L^2(\R)$, and the formal Gaussian variable $W((s-\cdot)_+^{-\beta})$ is not defined. Remark~\ref{rem:rankone} explains both this distinction and the additional uniform estimates that a regularization-based transfer would require. None of this precludes other representations, or applications of more general regularity criteria once their hypotheses are verified. What the present paper supplies is exactly such a verification, carried out for the entire class of singular weighted kernels in the form of Definition~\ref{def:admissible}, together with the uniformity over compact families that the induction requires and that the conclusions inherit.

A different regularization mechanism is established by Herry, Malicet and Poly~\cite[Theorem~3]{HMP2024}. For a sequence in a fixed Wiener chaos converging in law to a nondegenerate Gaussian, they obtain asymptotic bounds on negative Malliavin gradient moments of every order. Their work also provides local uniform estimates near the Gaussian law and uses induction on the chaos degree. Our theorem addresses a different hypothesis: fixed non-Gaussian weighted Hermite laws, without a Gaussian-limit assumption. The distinction is not the number of inverse moments, which is arbitrary in both settings, but the source of nondegeneracy and the family over which it is uniform. A related sphere discretization appears in~\cite[Proposition~31]{HMP2024}; we record the form needed for restricted gradients in Lemma~\ref{lem:net}.

General polynomial inequalities, such as the Carbery--Wright estimate~\cite{CarberyWright2001}, give degree-dependent small-ball bounds without this kernel-specific structure. The present work identifies and proves the stronger nondegeneracy property for the entire weighted Hermite class. Its contribution is the invariant compact-family hypothesis and its verification for these singular kernels, together with the resulting uniform inverse moments and the exact weighted-vector criterion.

\medskip
\noindent\emph{Consequences and scope.}
Indicator weights give all finite-dimensional Hermite-process distributions directly from the weighted-vector theorem, as well as vectors of non-overlapping increments. This settles the higher-order smoothness issue raised in~\cite[Remark~7.1]{LNNTAC} and~\cite[p.~2]{LNNTRosenblatt}: to the best of our knowledge, inverse-moment and Schwartz-density conclusions were not previously available for Hermite processes of order $q\ge4$. At $q=2,3$ the same argument recovers the known qualitative conclusions independently, and $q=1$ is the deterministic base step; Remark~\ref{rem:allorders} compares the orders.

General weights are not a technical generalization. By Corollary~\ref{cor:wiener}, $G_q(w)$ is, up to the normalizing constant, the Wiener integral $\int_0^Tw(s)\dd Z_s^{H,q}$ with respect to the Hermite process. Theorems~\ref{thm:mainuniform} and~\ref{thm:mainvector} therefore state that every such integral with a nonzero bounded deterministic weight is Malliavin nondegenerate, and that a finite family of them has a Schwartz density exactly when the weights are linearly independent. Compact weight families cover parametrized models: the Hermite Ornstein--Uhlenbeck weights $w_\theta(s)=e^{-\theta(t-s)}$ with $\theta$ in a compact set form an admissible family, so the density bounds hold uniformly in the parameter. Appendix~\ref{app:four} gives an alternative fourth-order proof through quadratic forms, connecting the general induction with the order-three method.

The uniform Schwartz bounds also supply a direct quantitative consequence. For each compact family $\cK$ as above, Proposition~\ref{prop:uniformvectors}, with one component, yields
\[
 \sup_{w\in\cK}\sup_{z\in\R}
 \Prob\{z<G_r(w)\le z+\varepsilon\}
 \le C_{\cK}\varepsilon,\qquad \varepsilon>0.
\]
In the same way, an $n$-dimensional family $\mathcal V$ as in Proposition~\ref{prop:uniformvectors} gives
\[
 \sup_{(w_1,\ldots,w_n)\in\mathcal V}\ \sup_{z\in\R^n}
 \Prob\{X\in B(z,\varepsilon)\}\le C_{\mathcal V}\varepsilon^n,
 \qquad \varepsilon>0,
\]
for $X=(G_r(w_1),\ldots,G_r(w_n))$, where $B(z,\varepsilon)$ is the Euclidean ball. The fixed order and kernel parameters are absorbed into the constants, which are independent of the location $z$. Estimates of this kind are the regularity input required by smoothing arguments and by Kolmogorov-distance bounds with a non-Gaussian target, where a bounded limiting density is needed; the Schwartz decay likewise licenses Fourier inversion and differentiation under the integral sign for these laws.

Uniformity concerns admissible compact families of weights at fixed order and kernel parameter; Remark~\ref{rem:scopeuniformity} describes the boundaries of this uniformity.

Section~\ref{sec:main} states the results and fixes the normalization. Section~\ref{sec:deterministic} establishes the deterministic properties of the weighted kernels. Section~\ref{sec:uniform} develops the uniform probabilistic estimates. The induction and the finite-dimensional consequences are proved in Sections~\ref{sec:induction} and~\ref{sec:vectors}, respectively.

\section{Framework and main results}\label{sec:main}

\subsection{Gaussian and Malliavin notation}
Let $W=\{W(h):h\in\Hh\}$ be an isonormal Gaussian process over the real Hilbert space $\Hh=L^2(\R)$, and let $I_r$ denote the multiple Wiener--It\^o integral of order $r$, normalized by
\[
 \E[I_r(f)I_r(g)]=r!\,\ip{f}{g}_{\Hh^{\otimes r}}
\]
for symmetric kernels $f,g$. We write $D$ for the Malliavin derivative and $\delta$ for its adjoint, the divergence operator. For a Hilbert space $E$, $k\ge0$ and $p\ge2$, we use the norm
\[
 \norm{V}_{k,p;E}
 =\left(\sum_{j=0}^k\E\norm{D^jV}_{\Hh^{\otimes j}\otimes E}^p\right)^{1/p}
\]
on $\mathbb D^{k,p}(E)$, with the scalar notation $\norm{V}_{k,p}$ when $E=\R$. Also,
\[
 \mathbb D^\infty=\bigcap_{k\ge0}\bigcap_{p\ge2}\mathbb D^{k,p}.
\]
These spaces, the multiple-integral derivative formula, and the divergence estimates used below are standard; see~\cite{Nualart2006}.

For $T>0$, identify $\HT=L^2(0,T)$ with its zero extension in $\Hh$ and define
\[
 \Dt V=\ind_{(0,T)}DV.
\]
For a vector $X\in(\mathbb D^{1,2})^n$, let
\begin{equation}\label{eq:restrictedmatrix}
 \Gamma_X^T=\bigl(\ip{\Dt X_i}{\Dt X_j}_{\HT}\bigr)_{i,j}.
\end{equation}
The full matrix dominates this restricted matrix in the positive-semidefinite order:
\begin{equation}\label{eq:matrixorder}
 \Gamma_X\succeq\Gamma_X^T,
 \qquad
 \lambda_{\min}(\Gamma_X)\ge\lambda_{\min}(\Gamma_X^T).
\end{equation}
Integration by parts below always uses the full Gaussian space and the full Malliavin matrix. The restricted derivative is used only to establish lower bounds.

\subsection{Weighted kernels}
Fix an integer $q\ge1$, $T>0$, and
\begin{equation}\label{eq:parameters}
 0<\alpha<\frac1q,\qquad
 \beta=\frac{1+\alpha}{2},\qquad
 b_\alpha=B\left(\frac{1-\alpha}{2},\alpha\right),
\end{equation}
where $B$ denotes the beta function. For a real bounded weight $w\in L^\infty(0,T)$ and $1\le r\le q$, put
\begin{equation}\label{eq:kernel}
 \begin{split}
 K_r[w](x_1,\ldots,x_r)
 &=\int_0^T w(s)\prod_{j=1}^r(s-x_j)_+^{-\beta}\dd s,\\
 G_r(w)&=I_r(K_r[w]).
 \end{split}
\end{equation}
Here $u_+^{-\beta}$ means $u^{-\beta}$ for $u>0$ and zero for $u\le0$. The kernels are understood up to Lebesgue-null sets; their square integrability is proved in Lemma~\ref{lem:kernels}. No normalizing constant is included in $G_r$.

\begin{theorem}[Uniform negative moments]\label{thm:mainuniform}
Let $\cK$ be a nonempty compact subset of $L^\infty(0,T)$ in the norm topology, consisting of real weights and not containing the zero element. For every $1\le r\le q$ and every $p>0$,
\begin{equation}\label{eq:mainuniform}
 \sup_{w\in\cK}\E\bigl[\norm{\Dt G_r(w)}_{\HT}^{-p}\bigr]<\infty.
\end{equation}
Equivalently, the following estimates hold for every $N>0$:
\begin{equation}\label{eq:mainsmallball}
 \sup_{w\in\cK}\Prob\{\norm{\Dt G_r(w)}_{\HT}\le\varepsilon\}
 \le C_N\varepsilon^N,\qquad 0<\varepsilon\le1.
\end{equation}
The constants may depend on $\cK$, $T$, $\alpha$, $r$ and the indicated exponent.
\end{theorem}

\begin{theorem}[Criterion for weighted vectors]\label{thm:mainvector}
Let $w_1,\ldots,w_n\in L^\infty(0,T)$ be real and linearly independent as equivalence classes modulo equality almost everywhere. For every $1\le r\le q$, the vector
\[
 X=(G_r(w_1),\ldots,G_r(w_n))
\]
satisfies, for every $p>0$,
\begin{equation}\label{eq:mainvector}
 \E[\lambda_{\min}(\Gamma_X^T)^{-p}]<\infty,
 \qquad
 \E[(\det\Gamma_X)^{-p}]<\infty.
\end{equation}
It admits a density in $\Sch(\R^n)$. Conversely, if the weights are linearly dependent, the law of $X$ is supported by a proper linear hyperplane and has no density with respect to $n$-dimensional Lebesgue measure.
\end{theorem}

\subsection{Hermite processes}
For an integer $q\ge1$ and $H\in(1/2,1)$, set
\begin{equation}\label{eq:hermiteparameters}
 \alpha=\frac{2(1-H)}q,\qquad
 \beta=\frac12+\frac{1-H}q,\qquad
 c_{H,q}=\left(\frac{H(2H-1)}{q!b_\alpha^q}\right)^{1/2}.
\end{equation}
The Hermite process is given, for $t\ge0$, by
\begin{equation}\label{eq:hermiteprocess}
 Z_t^{H,q}
 =c_{H,q}I_q\left(\int_0^t\prod_{j=1}^q(s-x_j)_+^{-\beta}\dd s\right).
\end{equation}
This agrees with the usual representation with $H_0=1+(H-1)/q$ and exponent $H_0-3/2$; see~\cite{Taqqu1979,LNNTAC}. Since $q\alpha=2-2H$, the energy identity \eqref{eq:energy} below gives
\[
 \E[(Z_1^{H,q})^2]
 =c_{H,q}^2\,q!\,b_\alpha^q\int_0^1\!\!\int_0^1|s-t|^{-q\alpha}\dd s\dd t
 =c_{H,q}^2\,q!\,b_\alpha^q\,\frac{2}{(2H-1)(2H)}=1,
\]
which is the normalization chosen in \eqref{eq:hermiteparameters}. With $T\ge t$,
\begin{equation}\label{eq:indicatorrepresentation}
 Z_t^{H,q}=c_{H,q}G_q(\ind_{(0,t)}).
\end{equation}

\begin{theorem}[Malliavin smoothness at every Hermite order]\label{thm:hermite}
Let $q\ge1$ be a fixed integer and $H\in(1/2,1)$. Then, for every $p>0$,
\begin{equation}\label{eq:hermiteonetime}
 \E\bigl[\norm{DZ_1^{H,q}}_{L^2(0,1)}^{-p}\bigr]<\infty;
\end{equation}
in particular $\E[\norm{DZ_1^{H,q}}_{\Hh}^{-p}]<\infty$, since $\norm{DZ_1^{H,q}}_{\Hh}\ge\norm{DZ_1^{H,q}}_{L^2(0,1)}$.
For every $n\ge1$ and every set of distinct positive times $t_1,\ldots,t_n$, the vector
\[
 \mathbf Z=(Z_{t_1}^{H,q},\ldots,Z_{t_n}^{H,q})
\]
has all negative moments of its Malliavin determinant and admits a density in $\Sch(\R^n)$. The same assertions hold for every finite vector
\[
 (Z_{b_1}^{H,q}-Z_{a_1}^{H,q},\ldots,Z_{b_n}^{H,q}-Z_{a_n}^{H,q})
\]
with $0\le a_1<b_1\le a_2<b_2\le\cdots\le a_n<b_n$.
\end{theorem}

\begin{corollary}[Wiener integrals with respect to a Hermite process]\label{cor:wiener}
Let $q\ge1$, $H\in(1/2,1)$ and $T>0$. For every real $w\in L^\infty(0,T)$,
\begin{equation}\label{eq:wienerintegral}
 c_{H,q}G_q(w)=\int_0^Tw(s)\dd Z_s^{H,q},
\end{equation}
the Wiener integral with respect to the Hermite process. Consequently, every such integral with $w\ne0$ has all negative moments of its Malliavin gradient norm; a finite family of them has a Schwartz density precisely when the weights are linearly independent; and both conclusions are uniform over admissible compact families of weights, in the sense of Theorem~\ref{thm:mainuniform} and Proposition~\ref{prop:uniformvectors}.
\end{corollary}
\begin{proof}
For $w=\ind_{(a,b)}$ with $0\le a<b\le T$, \eqref{eq:wienerintegral} is \eqref{eq:indicatorrepresentation} together with the linearity of $G_q$, which is the definition of the Wiener integral for step integrands ~\cite{MaejimaTudor2007}. Both sides of \eqref{eq:wienerintegral} are linear in $w$. By the energy identity \eqref{eq:energy} and Schur's test, using $q\alpha<1$ and $\sup_s\int_0^T|s-t|^{-q\alpha}\dd t<\infty$, the map $w\mapsto K_q[w]$ is bounded from $L^2(0,T)$ into $L^2(\R^q)$; with the isometry of multiple integrals, both sides are therefore continuous in $w$ for the $L^2(0,T)$ norm, and step functions are dense. The remaining assertions are Theorems~\ref{thm:mainuniform} and~\ref{thm:mainvector} applied with $r=q$, the positive factor $c_{H,q}$ being immaterial.
\end{proof}

\begin{remark}[Comparison with earlier orders]\label{rem:allorders}
The proof of Theorem~\ref{thm:hermite} given in Section~\ref{sec:vectors} follows directly from Theorems~\ref{thm:mainuniform} and~\ref{thm:mainvector}, which are based on the same induction scheme for every fixed finite chaos order. For $q=1$ the Malliavin derivatives are deterministic and the theorem is the classical nondegeneracy of fractional Brownian motion; here it is contained in the first-chaos case of Proposition~\ref{prop:induction}. For $q=2$ and $q=3$ the qualitative conclusions were obtained in~\cite{LNNTRosenblatt} and~\cite{Moldavskaya2026} by methods specific to the second and third chaos; Theorem~\ref{thm:hermite} provides an independent proof of these conclusions, but not of the quantitative density and grid-uniform estimates established there. For $q\ge4$ it gives the higher-order extension discussed in Section~\ref{sec:introduction}.
\end{remark}

Theorems~\ref{thm:mainuniform} and~\ref{thm:mainvector} also imply uniform Schwartz bounds for suitable compact families of weighted vectors; the precise version is Proposition~\ref{prop:uniformvectors}. All linear-independence assumptions in this paper concern deterministic weights. No independence of random coordinates or increments is assumed.

\section{Deterministic structure and stable weight families}\label{sec:deterministic}

\subsection{Energy and contraction identities}
\begin{lemma}[Kernel estimates]\label{lem:kernels}
Under \eqref{eq:parameters}, $K_r[w]$ is a symmetric element of $L^2(\R^r)$. For real $w,v\in L^\infty(0,T)$,
\begin{equation}\label{eq:energy}
 \ip{K_r[w]}{K_r[v]}_{L^2(\R^r)}
 =b_\alpha^r\int_0^T\int_0^T w(s)v(t)|s-t|^{-r\alpha}\dd s\dd t.
\end{equation}
In particular,
\begin{equation}\label{eq:kernelbound}
 \norm{K_r[w]}_2^2
 \le b_\alpha^r\norm w_\infty^2
 \frac{2T^{2-r\alpha}}{(1-r\alpha)(2-r\alpha)}.
\end{equation}
For every family $\mathcal W$ of real weights bounded in $L^\infty(0,T)$, every $k\ge0$ and $p\ge2$,
\begin{equation}\label{eq:uniformsobolev}
 \sup_{w\in\mathcal W}\norm{G_r(w)}_{k,p}<\infty,
\end{equation}
with a bound depending only on $\sup_{w\in\mathcal W}\norm w_\infty$, $T$, $\alpha$, $r$, $k$ and $p$.
\end{lemma}
\begin{proof}
For $s\ne t$, a change of variables in the beta integral gives
\begin{equation}\label{eq:betaidentity}
 \int_\R(s-x)_+^{-\beta}(t-x)_+^{-\beta}\dd x
 =b_\alpha|s-t|^{-\alpha}.
\end{equation}
Apply Tonelli's theorem with the nonnegative weights $|w|$ and $|v|$. The resulting double integral is finite since $r\alpha<1$. In particular, $K_r[|w|]$ is finite almost everywhere and square integrable. Thus \eqref{eq:kernel} defines $K_r[w]$ almost everywhere, and Fubini's theorem for signed weights yields \eqref{eq:energy}. Since
\[
 \int_0^T\int_0^T|s-t|^{-r\alpha}\dd s\dd t
 =\frac{2T^{2-r\alpha}}{(1-r\alpha)(2-r\alpha)},
\]
we obtain \eqref{eq:kernelbound}.

For $0\le j\le r$, the multiple-integral derivative formula gives
\begin{equation}\label{eq:derivativesecondmoment}
 \E\norm{D^jG_r(w)}_{\Hh^{\otimes j}}^2
 =\frac{(r!)^2}{(r-j)!}\norm{K_r[w]}_2^2,
\end{equation}
while derivatives of order larger than $r$ vanish. Hypercontractivity applies also to Hilbert-valued chaos variables. For completeness, it follows from scalar hypercontractivity by expanding a Hilbert-valued variable in an orthonormal basis and using
\[
 \left\|\sum_j |V_j|^2\right\|_{L^{p/2}}
 \le\sum_j\norm{V_j}_{L^p}^2,\qquad p\ge2,
\]
first for finite sums and then by approximation. Applying this to \eqref{eq:derivativesecondmoment} proves \eqref{eq:uniformsobolev}; see~\cite[Chapter 1]{Nualart2006} for scalar hypercontractivity and the chaos calculus.
\end{proof}

\begin{remark}[Singular kernels and regularization]\label{rem:rankone}
For fixed $s$, put $\phi_s(x)=(s-x)_+^{-\beta}$. Since $2\beta=1+\alpha>1$, the singularity at $x=s$ prevents $\phi_s$ from belonging to $L^2(\R)$. Thus $W(\phi_s)$ is not a real-valued Gaussian random variable in the given isonormal construction, and \eqref{eq:kernel} is not a Bochner integral of $L^2$-valued rank-one tensors. This rules out the direct interpretation as an ordinary integral polynomial of this formal pointwise field, not every possible polynomial representation. The absolute-weight Tonelli argument in Lemma~\ref{lem:kernels} constructs the kernels without such a field.

To make the regularization issue explicit, let
\[
 \phi_s^\delta(x)=(s-x)_+^{-\beta}\ind_{\{x<s-\delta\}},
 \qquad Y_s^\delta=W(\phi_s^\delta),\qquad \delta>0.
\]
These Gaussian variables are well defined, but
\[
 \sigma_\delta^2:=\E[(Y_s^\delta)^2]
 =\int_\delta^\infty u^{-1-\alpha}\dd u
 =\frac{\delta^{-\alpha}}{\alpha}\longrightarrow\infty.
\]
For the truncated kernels $K_r^\delta[w]=\int_0^T w(s)(\phi_s^\delta)^{\otimes r}\dd s$, the rank-one integral is a legitimate Bochner integral. Moreover, $K_r^\delta[w]\to K_r[w]$ in $L^2(\R^r)$: pointwise convergence holds almost everywhere by dominated convergence in $s$, and the bound $|K_r^\delta[w]|\le K_r[|w|]\in L^2$ permits dominated convergence in the kernel variables. The isometry of multiple integrals gives $I_r(K_r^\delta[w])\to G_r(w)$ in $L^2(\Omega)$. If $H_r$ denotes the probabilists' Hermite polynomial, the regularized integral is
\[
 I_r(K_r^\delta[w])
 =\int_0^T w(s)\sigma_\delta^r
       H_r\!\left(\frac{Y_s^\delta}{\sigma_\delta}\right)\dd s
 \quad\text{in }L^2(\Omega).
\]
Consequently, an ordinary polynomial representation of the approximations has a $\delta$-dependent Gaussian field and, for $r\ge2$, variance-dependent centering terms. Smoothness of each approximation, even if established separately, would not by itself give the uniform inverse moments needed for a limiting argument. The variance divergence alone does not establish deterioration of Malliavin nondegeneracy or failure of uniform inverse-moment bounds. The proof below establishes those moments directly for the limiting singular kernels.
\end{remark}

For real $h\in C_c^\infty(0,T)$, extended by zero to $\R$, define
\begin{equation}\label{eq:ah}
 a_h(s)=\int_0^s h(t)(s-t)^{-\beta}\dd t,\qquad 0<s<T.
\end{equation}
The elementary bound
\begin{equation}\label{eq:ahbound}
 \norm{a_h}_\infty
 \le\frac{T^{1-\beta}}{1-\beta}\norm h_\infty
\end{equation}
will be used repeatedly.

\begin{lemma}[Directional differentiation]\label{lem:directional}
For $2\le r\le q$, real $w\in L^\infty(0,T)$ and $h\in C_c^\infty(0,T)$,
\begin{equation}\label{eq:directional}
 D_hG_r(w):=\ip{\Dt G_r(w)}h_{\HT}
 =rG_{r-1}(wa_h).
\end{equation}
Since $h$ vanishes outside $(0,T)$, the left-hand side of \eqref{eq:directional} coincides with the full-space directional derivative $\ip{DG_r(w)}{h}_{\Hh}$. At the first chaos level,
\begin{equation}\label{eq:firstderivative}
 (\Dt G_1(w))(x)=\int_x^T w(s)(s-x)^{-\beta}\dd s,
 \qquad 0<x<T.
\end{equation}
\end{lemma}
\begin{proof}
The identity $D_hI_r(f)=rI_{r-1}(f\otimes_1h)$ reduces the assertion to the contraction of $K_r[w]$ with $h$. That contraction is square integrable and has norm at most $\norm{K_r[w]}_2\norm h_2$. Applying Tonelli with $|w|$ and $|h|$, and then Fubini for the signed integrand, identifies it with $K_{r-1}[wa_h]$. This proves \eqref{eq:directional}. For $r=1$, $DI_1(K_1[w])=K_1[w]$, and restriction to $(0,T)$ gives \eqref{eq:firstderivative}.
\end{proof}

\subsection{Analyticity after the support of a direction}
Fix $0<\tau<T$ throughout this subsection.

\begin{lemma}[Analytic-tail nonvanishing]\label{lem:tail}
If $0\ne h\in C_c^\infty(0,\tau)$ is real, the function
\[
 a_h(s)=\int_0^\tau h(t)(s-t)^{-\beta}\dd t,\qquad s>\tau,
\]
is real analytic and not identically zero. In particular, its zero set in $(\tau,T)$ has Lebesgue measure zero.
\end{lemma}
\begin{proof}
For $t\in[0,\tau]$ the function $z\mapsto(z-t)^{-\beta}$, defined with the principal branch, is holomorphic on the half-plane $\{\Re z>\tau\}$, and for $z$ in a compact subset of this half-plane it is bounded uniformly in $t$. Differentiation under the integral sign (or Morera's theorem together with Fubini's theorem) therefore shows that $a_h$ extends to a holomorphic function on $\{\Re z>\tau\}$; in particular $a_h$ is real analytic on $(\tau,\infty)$. If $a_h$ vanished on an open subinterval of $(\tau,\infty)$, the identity theorem on the connected set $(\tau,\infty)$ would imply $a_h(s)=0$ for every $s>\tau$.

For such $s$, the binomial expansion is uniformly convergent for $t\in[0,\tau]$ and gives
\begin{equation}\label{eq:momentexpansion}
 s^\beta a_h(s)
 =\sum_{n=0}^\infty\frac{(\beta)_n}{n!s^n}
       \int_0^\tau t^nh(t)\dd t,
\end{equation}
where $(\beta)_n=\beta(\beta+1)\cdots(\beta+n-1)$ and $(\beta)_0=1$. In the variable $z=1/s$, the right-hand side is a power series for $|z|<1/\tau$. Its vanishing for $0<z<1/\tau$ forces all coefficients to be zero. Since $(\beta)_n\ne0$, every polynomial moment of $h$ vanishes. Density of polynomials in $L^2(0,\tau)$ gives $h=0$, a contradiction. The zeros of a nonzero real-analytic function on an interval are isolated and hence form a null set.
\end{proof}

\begin{definition}[Admissible compact family]\label{def:admissible}
For fixed $T$ and $\tau$, a nonempty family $\cK$ of real weights is admissible if it is compact in $L^\infty(0,T)$ and
\begin{equation}\label{eq:admissible}
 \int_\tau^T|w(s)|^2\dd s>0\qquad\text{for every }w\in\cK.
\end{equation}
\end{definition}
The weights themselves need not be analytic, continuous, or nonzero almost everywhere. Condition~\eqref{eq:admissible} only asks for nonvanishing on a set of positive measure after the common time $\tau$.

\begin{lemma}[Stability under directional differentiation]\label{lem:stable}
Let $\cK$ be admissible and let $h_1,\ldots,h_m\in C_c^\infty(0,\tau)$ be real and orthonormal in $L^2(0,T)$. For $u\in\Sphere^{m-1}$ put $h_u=\sum_i u_ih_i$. Then
\begin{equation}\label{eq:descendants}
 \cK^{[h]}=\{wa_{h_u}:w\in\cK,\ u\in\Sphere^{m-1}\}
\end{equation}
is admissible for the same $T$ and $\tau$.
\end{lemma}
\begin{proof}
The map $(w,u)\mapsto wa_{h_u}$ is continuous from $\cK\times\Sphere^{m-1}$ into $L^\infty(0,T)$: by linearity of $h\mapsto a_h$ and \eqref{eq:ahbound},
\[
 \norm{wa_{h_u}-w'a_{h_{u'}}}_\infty
 \le\norm{w-w'}_\infty\norm{a_{h_u}}_\infty
   +\norm{w'}_\infty\frac{T^{1-\beta}}{1-\beta}\,|u-u'|\Bigl(\sum_{i=1}^m\norm{h_i}_\infty^2\Bigr)^{1/2}.
\]
Its domain is compact, so its image is compact. Orthonormality implies $h_u\ne0$ for every unit $u$. Lemma~\ref{lem:tail} gives $a_{h_u}\ne0$ almost everywhere on $(\tau,T)$. Multiplication by $a_{h_u}$ cannot annihilate the positive-measure set on which $w$ is nonzero. Thus every member of \eqref{eq:descendants} satisfies \eqref{eq:admissible}.
\end{proof}

\begin{lemma}[A common terminal interval]\label{lem:commontail}
Every nonempty norm-compact family $\cK\subset L^\infty(0,T)$ not containing zero is admissible for some $\tau\in(0,T)$.
\end{lemma}
\begin{proof}
Continuity of the embedding into $L^2(0,T)$ and compactness give
\[
 d:=\inf_{w\in\cK}\norm w_2>0,
 \qquad M:=\sup_{w\in\cK}\norm w_\infty<\infty.
\]
Choose $0<\tau<\min\{T,d^2/(2M^2)\}$. Then
\[
 \int_\tau^T|w(s)|^2\dd s
 \ge d^2-\tau M^2>d^2/2
 \qquad\text{for every }w\in\cK.
\]
\end{proof}

\section{Uniform Malliavin estimates}\label{sec:uniform}

We isolate two probabilistic facts with their uniformity made explicit. The first is a compactness estimate on a finite-dimensional sphere. The second is a uniform version of the usual Malliavin density argument. In both statements the index set may be arbitrary; no joint measurability or continuity in the index is needed. Suprema over indices remain outside probabilities and expectations.

\subsection{From individual directions to the smallest eigenvalue}
The following estimate uses the standard discretization of a finite-dimensional sphere. A related discretization appears in Herry, Malicet and Poly~\cite[Proposition~31]{HMP2024}. We state and prove the restricted-gradient version with uniformity over an arbitrary index family.
\begin{lemma}[Sphere-net estimate]\label{lem:net}
Let $Y_\theta=(Y_{\theta,1},\ldots,Y_{\theta,m})\in(\mathbb D^{1,2})^m$, $\theta\in\Theta$. Suppose that, for every $N>0$,
\begin{equation}\label{eq:netsmallballassumption}
 \sup_{\theta\in\Theta,\ u\in\Sphere^{m-1}}
 \Prob\{\norm{\Dt(u^\top Y_\theta)}_{\HT}\le\varepsilon\}
 \le C_N\varepsilon^N,\qquad 0<\varepsilon\le1,
\end{equation}
and that
\begin{equation}\label{eq:netLassumption}
 L_\theta:=\left(\sum_{i=1}^m\norm{\Dt Y_{\theta,i}}_{\HT}^2\right)^{1/2}
 \quad\text{satisfies}\quad
 \sup_\theta\E L_\theta^p<\infty\quad\text{for all }p>0.
\end{equation}
Then for every $p>0$,
\begin{equation}\label{eq:neteigenvalue}
 \sup_\theta\E[\lambda_{\min}(\Gamma_{Y_\theta}^T)^{-p}]<\infty.
\end{equation}
The same assertion holds for the full Malliavin matrices.
\end{lemma}
\begin{proof}
For each fixed $\theta$, choose Hilbert-valued representatives of the finitely many gradients and use their linear combinations for every $u$. Then
\[
 X_\theta(u):=\norm{\Dt(u^\top Y_\theta)}_{\HT}
 \quad\text{satisfies}\quad
 |X_\theta(u)-X_\theta(v)|\le L_\theta|u-v|.
\]
In particular, for every $\omega$ the function $u\mapsto X_\theta(u)(\omega)$ is continuous on the compact sphere and attains its minimum, and $X_{\theta,*}=\min_{|u|=1}X_\theta(u)$ is measurable; the infimum may equivalently be taken over a fixed countable dense subset of the sphere.

For $0<\varepsilon\le1/2$, take a deterministic $\varepsilon^2$-net $\mathcal N_\varepsilon$ of $\Sphere^{m-1}$ with at most $C_m\varepsilon^{-2m}$ points. Such a net is obtained, for example, by taking a maximal $\varepsilon^2$-separated subset and comparing the volumes of the disjoint balls of radius $\varepsilon^2/2$ centred at its points with the volume of the ball of radius $1+\varepsilon^2/2$ in $\R^m$. On the event
\[
 \{X_{\theta,*}\le\varepsilon,\ L_\theta\le\varepsilon^{-1}\},
\]
choose a minimizer $u^*$ and $v\in\mathcal N_\varepsilon$ with $|u^*-v|\le\varepsilon^2$; then $X_\theta(v)\le X_\theta(u^*)+L_\theta\varepsilon^2\le2\varepsilon$. A union bound over the net, assumption \eqref{eq:netsmallballassumption} at the level $2\varepsilon\le1$, and Markov's inequality for $L_\theta$ therefore give, for arbitrary $A,B>0$,
\begin{equation}\label{eq:netestimate}
 \sup_\theta\Prob\{X_{\theta,*}\le\varepsilon\}
 \le C_{m,A}\varepsilon^{A-2m}
       +\varepsilon^B\sup_\theta\E L_\theta^B,
 \qquad0<\varepsilon\le\tfrac12,
\end{equation}
where $C_{m,A}=2^AC_mC_A$. For any desired exponent $N$, choose $A>N+2m$ and $B>N$; for $1/2<\varepsilon\le1$ the bound $\Prob\{X_{\theta,*}\le\varepsilon\}\le1\le2^N\varepsilon^N$ is trivial. This proves small-ball estimates of every order for $X_{\theta,*}$, uniformly in $\theta$.

For a positive random variable $X$ and $s>0$, the elementary bound
\begin{equation}\label{eq:layercake}
 \E X^{-s}\le1+s\int_0^1 t^{-s-1}\Prob\{X<t\}\dd t
\end{equation}
shows that these estimates imply all inverse moments. Since
\[
 \lambda_{\min}(\Gamma_{Y_\theta}^T)=X_{\theta,*}^2,
\]
we use an exponent larger than $2p$ in \eqref{eq:netestimate} to obtain \eqref{eq:neteigenvalue}. The full-matrix assertion follows from \eqref{eq:matrixorder}. No independence among the net events is used.
\end{proof}

\subsection{Uniform inverse matrices and integration-by-parts weights}
\begin{lemma}[Uniform Schwartz-density bounds]\label{lem:density}
Let $Y_\theta\in(\mathbb D^\infty)^m$, $\theta\in\Theta$, and suppose that, for all $k\ge0$, $p\ge2$ and $1\le i\le m$,
\begin{equation}\label{eq:densitysobolevassumption}
 \sup_\theta\norm{Y_{\theta,i}}_{k,p}<\infty.
\end{equation}
Suppose also that, for every $p>0$,
\begin{equation}\label{eq:densityinverseassumption}
 \sup_\theta\E[\lambda_{\min}(\Gamma_{Y_\theta})^{-p}]<\infty.
\end{equation}
Then $Y_\theta$ has a density $p_\theta\in\Sch(\R^m)$, and for every pair of multi-indices $\mu,\nu\in\N_0^m$,
\begin{equation}\label{eq:uniformSchwartz}
 \sup_{\theta\in\Theta}\sup_{x\in\R^m}
 |x^\mu\partial^\nu p_\theta(x)|<\infty.
\end{equation}
In particular, $\sup_\theta\norm{p_\theta}_\infty<\infty$.
\end{lemma}
\begin{proof}
The lemma formulates standard Malliavin integration by parts with the parameter control needed to propagate the induction over compact families. We give the details of the uniform estimates, since separate smoothness of the laws would not imply the bound required there.

\emph{Step 1: inverse-matrix Sobolev estimates.}
Write $\Gamma_\theta=\Gamma_{Y_\theta}$ and $\lambda_\theta=\lambda_{\min}(\Gamma_\theta)$. By \eqref{eq:densityinverseassumption}, $\lambda_\theta>0$ almost surely. For $0<\eta\le1$, set
\[
 A_{\theta,\eta}=(\Gamma_\theta+\eta I_m)^{-1}.
\]
The entries of $\Gamma_\theta$ have all Sobolev norms uniformly bounded, by the product rule and \eqref{eq:densitysobolevassumption}. Moreover,
\[
 \norm{A_{\theta,\eta}}_{\mathrm{op}}\le\lambda_\theta^{-1}.
\]
For fixed $\eta$, membership of $A_{\theta,\eta}$ in $\mathbb D^\infty$ follows without any assumption that $\Gamma_\theta$ is bounded away from singularity. Indeed, $\det(\Gamma_\theta+\eta I_m)\ge\eta^m$ almost surely. Choose a smooth function $f_\eta$ on $\R$, with bounded derivatives of every order, equal to $t^{-1}$ for $t\ge\eta^m/2$. Then
\[
 A_{\theta,\eta}
 =\adj(\Gamma_\theta+\eta I_m)
       f_\eta\bigl(\det(\Gamma_\theta+\eta I_m)\bigr).
\]
The entries of $\adj(\Gamma_\theta+\eta I_m)$ and $\det(\Gamma_\theta+\eta I_m)$ are polynomials in the entries of $\Gamma_\theta$, hence belong to $\mathbb D^\infty$ by the Leibniz rule and H\"older's inequality; $f_\eta(\det(\Gamma_\theta+\eta I_m))\in\mathbb D^\infty$ by repeated application of the chain rule~\cite[Proposition 1.2.3]{Nualart2006}, together with the Leibniz rule and H\"older's inequality; and the product is again in $\mathbb D^\infty$. Differentiating the identity $A_{\theta,\eta}(\Gamma_\theta+\eta I_m)=I_m$ by the Leibniz rule gives
\begin{equation}\label{eq:inverseD}
 DA_{\theta,\eta}
 =-A_{\theta,\eta}(D\Gamma_\theta)A_{\theta,\eta}.
\end{equation}
This is the uniform counterpart of the argument in~\cite[Lemma 2.1.6]{Nualart2006}.
Repeated differentiation produces finite sums of products of inverse matrices and derivatives of $\Gamma_\theta$. In compatible Hilbert tensor norms, for $k\ge1$ this yields
\begin{equation}\label{eq:inverseDbound}
 \norm{D^kA_{\theta,\eta}}
 \le C_{k,m}\sum_{\ell=1}^k\lambda_\theta^{-\ell-1}
 \sum_{\substack{j_1+\cdots+j_\ell=k\\j_1,\ldots,j_\ell\ge1}}
 \prod_{a=1}^{\ell}\norm{D^{j_a}\Gamma_\theta}.
\end{equation}
H\"older's inequality and the hypotheses show that every term on the right has every finite moment, uniformly in $\theta$ and $\eta$.

For fixed $\theta$, $A_{\theta,\eta}$ converges almost surely to $A_\theta=\Gamma_\theta^{-1}$. Each differentiated product also converges almost surely. The bounds just obtained give domination in each $L^p$, so the regularized matrices and all their derivatives converge in the corresponding $L^p$ spaces. Closedness of the Malliavin derivative gives $A_\theta\in\mathbb D^{k,p}$ entrywise and
\begin{equation}\label{eq:inverseuniformsob}
 \sup_\theta\norm{(A_\theta)_{ij}}_{k,p}<\infty
 \qquad\text{for all }k\ge0,\ p\ge2.
\end{equation}
This also justifies \eqref{eq:inverseD} and its iterates for the unregularized inverse.

\emph{Step 2: integration-by-parts weights.}
For a scalar $V\in\mathbb D^\infty$, define
\begin{equation}\label{eq:IBPweight}
 H_i(Y_\theta,V)
 =\delta\left(V\sum_{j=1}^m(A_\theta)_{ij}DY_{\theta,j}\right).
\end{equation}
We use the standard divergence Sobolev estimate
\begin{equation}\label{eq:meyer}
 \norm{\delta U}_{k,p}
 \le C_{k,p}\norm U_{k+1,p;\Hh},\qquad k\ge0,\ p\ge2,
\end{equation}
which follows from Meyer's inequalities; see~\cite[Proposition 1.5.7]{Nualart2006}. Combining \eqref{eq:meyer} with the product rule and H\"older's inequality gives
\begin{equation}\label{eq:IBPexplicitbound}
 \norm{H_i(Y_\theta,V)}_{k,p}
 \le C_{k,p,m}\norm V_{k+1,3p}
 \sum_{j=1}^m\norm{(A_\theta)_{ij}}_{k+1,3p}
                      \norm{Y_{\theta,j}}_{k+2,3p}.
\end{equation}
Thus a family of weights $V_\theta$ with all Sobolev norms uniformly bounded produces the same kind of family after one application of $H_i$. The property is preserved by any finite number of iterations.

For a sequence of indices $\mathbf i=(i_1,\ldots,i_k)$, define the iterated weight by successively applying $H_{i_1},\ldots,H_{i_k}$, starting from $V$. By duality between $D$ and $\delta$ and by $A_\theta\Gamma_\theta=I_m$, we have
\begin{equation}\label{eq:IBPidentity}
 \E[(\partial_{i_1}\cdots\partial_{i_k}\varphi)(Y_\theta)V]
 =\E[\varphi(Y_\theta)H_{\mathbf i}(Y_\theta,V)]
\end{equation}
for $\varphi\in C_b^\infty(\R^m)$. For one differentiation, this follows directly by pairing $D\varphi(Y_\theta)$ with the vector inside the divergence in \eqref{eq:IBPweight}; iteration proves the general formula. Complex exponentials are handled by their real and imaginary parts. Formula~\eqref{eq:IBPidentity} is the usual Malliavin integration-by-parts identity~\cite[Proposition 2.1.4]{Nualart2006}; estimate~\eqref{eq:IBPexplicitbound} records the uniformity needed here.

\emph{Step 3: Fourier estimates.}
Let $\chi_\theta(t)=\E e^{it^\top Y_\theta}$. If $|t|\ge1$, choose an index $i$ with $|t_i|\ge|t|/\sqrt m$. Since $\partial_i^k e^{it^\top y}=(it_i)^ke^{it^\top y}$, identity \eqref{eq:IBPidentity} with $\mathbf i=(i,\ldots,i)$ of length $k$, $V=1$, and $\varphi$ equal to the real and imaginary parts of $e^{it^\top y}$ gives
\[
 |t_i|^k|\chi_\theta(t)|
 \le2\,\E|H_{i,\ldots,i}(Y_\theta,1)|,
 \qquad\text{hence}\qquad
 |\chi_\theta(t)|\le2m^{k/2}|t|^{-k}\sup_{\theta}\E|H_{i,\ldots,i}(Y_\theta,1)|.
\]
The supremum is finite by Steps 1 and 2. Taking $k>m$ and using $|\chi_\theta|\le1$ on $\{|t|\le1\}$ proves a uniform $L^1(\R^m)$ bound on $\chi_\theta$, so Fourier inversion gives a bounded continuous density with $\sup_\theta\norm{p_\theta}_\infty\le(2\pi)^{-m}\sup_\theta\norm{\chi_\theta}_{L^1}$.

To obtain all Schwartz bounds, use $V_\theta=Y_\theta^\mu$. Polynomial weights have all Sobolev norms uniformly bounded by \eqref{eq:densitysobolevassumption}. Differentiation under expectation is justified by the corresponding moments, and
\[
 \partial_t^\mu\chi_\theta(t)
 =i^{|\mu|}\E[Y_\theta^\mu e^{it^\top Y_\theta}].
\]
The preceding integration-by-parts argument yields, for every $N\ge0$ and every $\mu$,
\begin{equation}\label{eq:CFSchwartz}
 \sup_\theta\sup_{t\in\R^m}
 (1+|t|)^N|\partial_t^\mu\chi_\theta(t)|<\infty.
\end{equation}
On $|t|\le1$ the bound follows from the moments of $Y_\theta$. Thus the characteristic functions form a bounded set in $\Sch(\R^m)$. Continuity of Fourier inversion on the Schwartz space proves \eqref{eq:uniformSchwartz}.
\end{proof}

\begin{remark}\label{rem:fullIBP}
Lemma~\ref{lem:net} may be applied to restricted gradients, whereas Lemma~\ref{lem:density} uses the full Malliavin inverse matrix. The comparison \eqref{eq:matrixorder} connects the two statements. No conditional or partial Malliavin integration-by-parts formula is being assumed.
\end{remark}

\section{Induction over the chaos order}\label{sec:induction}

The induction is carried by the admissible-family condition of Definition~\ref{def:admissible}: norm compactness and nonvanishing on a common terminal interval. Lemmas~\ref{lem:tail} and~\ref{lem:stable} preserve this condition under every family of directional contractions needed below. The assertion at the preceding level is therefore available for the entire descendant family, including all directions on the sphere. This is the structural closure that allows the probabilistic estimates of Section~\ref{sec:uniform} to be iterated.

\begin{proposition}[Induction on admissible families]\label{prop:induction}
Fix $q,T,\alpha,\beta$ as in \eqref{eq:parameters} and fix $0<\tau<T$. For every admissible family $\cK$ and every $1\le r\le q$, the estimates \eqref{eq:mainuniform} and \eqref{eq:mainsmallball} hold.
\end{proposition}
\begin{proof}
We prove the assertions at level $r$ for \emph{every} admissible family. Throughout the induction $\alpha$, $\beta$, $T$ and $\tau$ remain fixed.

\emph{Base case $r=1$.}
By \eqref{eq:firstderivative}, the restricted derivative is the deterministic function
\[
 (\Dt G_1(w))(x)=\int_x^T w(s)(s-x)^{-\beta}\dd s.
\]
For $\gamma>0$, define the right-sided fractional integral
\[
 (I_{T-}^\gamma w)(x)
 =\frac1{\Gamma(\gamma)}\int_x^T(s-x)^{\gamma-1}w(s)\dd s.
\]
Then $\Dt G_1(w)=\Gamma(1-\beta)I_{T-}^{1-\beta}w$. For $w\in L^1(0,T)$, Fubini's theorem and the beta integral give, for almost every $x\in(0,T)$,
\begin{align*}
 (I_{T-}^{\beta}I_{T-}^{1-\beta}w)(x)
 &=\frac{1}{\Gamma(\beta)\Gamma(1-\beta)}
   \int_x^T w(s)\left[
     \int_x^s(t-x)^{\beta-1}(s-t)^{-\beta}\dd t
   \right]\dd s\\
 &=\int_x^T w(s)\dd s=(I_{T-}^1w)(x).
\end{align*}
Tonelli's theorem applied to $|w|$ justifies the interchange. If $\Dt G_1(w)=0$ almost everywhere, applying $I_{T-}^{\beta}$ and using this identity gives $I_{T-}^1w=0$.
Thus $\int_x^T w(s)\dd s=0$ for almost every $x$, hence for every $x$ by continuity. This integral is absolutely continuous in $x$, so $w=0$ almost everywhere, contrary to admissibility.

The map $w\mapsto\Dt G_1(w)$ is continuous from $L^\infty(0,T)$ to $L^2(0,T)$, since
\[
 |(\Dt G_1(w))(x)|
 \le\frac{\norm w_\infty}{1-\beta}(T-x)^{1-\beta}.
\]
Compactness therefore yields
\begin{equation}\label{eq:basepositive}
 \delta_{\cK}:=\inf_{w\in\cK}\norm{\Dt G_1(w)}_{\HT}>0.
\end{equation}
Since $\Dt G_1(w)$ is deterministic, $\Prob\{\norm{\Dt G_1(w)}_{\HT}\le\varepsilon\}$ equals $0$ for $\varepsilon<\delta_{\cK}$ and is at most $1\le(\varepsilon/\delta_{\cK})^N$ for $\varepsilon\ge\delta_{\cK}$, so \eqref{eq:mainsmallball} holds with $C_N=\delta_{\cK}^{-N}$, and \eqref{eq:mainuniform} holds with $\sup_{w\in\cK}\E[\norm{\Dt G_1(w)}_{\HT}^{-p}]\le\delta_{\cK}^{-p}$. This proves both assertions at level one.

\emph{Induction step.}
Assume that both assertions hold at level $r-1$ for every admissible family. Let $\cK$ be an arbitrary admissible family at level $r$, and choose an arbitrary integer $m\ge1$. There exist real orthonormal functions
\[
 h_1,\ldots,h_m\in C_c^\infty(0,\tau);
\]
for example, one may normalize smooth nonzero functions with disjoint supports. Define
\begin{equation}\label{eq:directionalvector}
 Y_w=(D_{h_1}G_r(w),\ldots,D_{h_m}G_r(w)).
\end{equation}
For $u\in\Sphere^{m-1}$, Lemma~\ref{lem:directional} gives
\begin{equation}\label{eq:inductioncombination}
 u^\top Y_w=rG_{r-1}(wa_{h_u}).
\end{equation}
By Lemma~\ref{lem:stable}, the family of all weights $wa_{h_u}$, with $w\in\cK$ and $u$ on the sphere, is admissible for the same $T$ and $\tau$. The induction hypothesis \eqref{eq:mainsmallball} at level $r-1$, applied to this entire family, implies, for every $N>0$,
\begin{equation}\label{eq:inductionuniformdirections}
 \begin{split}
 \sup_{w\in\cK,\ |u|=1}
 \Prob\{\norm{\Dt(u^\top Y_w)}_{\HT}\le\varepsilon\}
 &=\sup_{w\in\cK,\ |u|=1}
 \Prob\{\norm{\Dt G_{r-1}(wa_{h_u})}_{\HT}\le\varepsilon/r\}\\
 &\le C_{m,N}\varepsilon^N,
 \qquad 0<\varepsilon\le1,
 \end{split}
\end{equation}
with $C_{m,N}=r^{-N}C_N$, where $C_N$ is the constant of the level-$(r-1)$ estimate for the family $\cK^{[h]}$; this constant depends on $m$ through $\cK^{[h]}$.

By Lemma~\ref{lem:kernels} and \eqref{eq:ahbound}, the weights $wa_{h_i}$ are bounded in $L^\infty(0,T)$ uniformly in $w\in\cK$ and $1\le i\le m$, so the coordinates $D_{h_i}G_r(w)=rG_{r-1}(wa_{h_i})$ of $Y_w$ have all Sobolev norms uniformly bounded. In particular, the random Lipschitz constants $L_w$ in Lemma~\ref{lem:net} have every positive moment uniformly in $w$. Lemma~\ref{lem:net}, applied with $\Theta=\cK$ and with hypothesis \eqref{eq:netsmallballassumption} supplied by \eqref{eq:inductionuniformdirections}, gives, first for the restricted matrices and then by \eqref{eq:matrixorder} for the full ones,
\[
 \sup_{w\in\cK}\E[\lambda_{\min}(\Gamma_{Y_w})^{-p}]<\infty
 \qquad\text{for every }p>0.
\]
Lemma~\ref{lem:density} now supplies joint densities $p_w$ of $Y_w$ with
\begin{equation}\label{eq:directionaldensitybound}
 \sup_{w\in\cK}\norm{p_w}_\infty\le C_m<\infty.
\end{equation}

Bessel's inequality is applied only to the first directional derivatives in \eqref{eq:directionalvector}:
\[
 |Y_w|^2=\sum_{i=1}^m|D_{h_i}G_r(w)|^2
 \le\norm{\Dt G_r(w)}_{\HT}^2.
\]
Hence $\{\norm{\Dt G_r(w)}_{\HT}\le\varepsilon\}\subset\{|Y_w|\le\varepsilon\}$. If $\omega_m$ denotes the volume of the Euclidean unit ball in $\R^m$, \eqref{eq:directionaldensitybound} implies
\begin{equation}\label{eq:finalbessel}
 \sup_{w\in\cK}\Prob\{\norm{\Dt G_r(w)}_{\HT}\le\varepsilon\}
 \le\sup_{w\in\cK}\int_{\{|y|\le\varepsilon\}}p_w(y)\dd y
 \le C_m\omega_m\varepsilon^m.
\end{equation}
Since $m$ was arbitrary, this proves \eqref{eq:mainsmallball}. For a prescribed $p>0$, choose $m>p$ in \eqref{eq:finalbessel} and apply \eqref{eq:layercake}; this gives \eqref{eq:mainuniform}. The induction is complete.
\end{proof}

\begin{proof}[Proof of Theorem~\ref{thm:mainuniform}]
By Lemma~\ref{lem:commontail}, the given compact family is admissible for some $\tau\in(0,T)$. Proposition~\ref{prop:induction} then applies. Conversely, the existence of all uniform inverse moments implies all the small-ball bounds by Markov's inequality, while \eqref{eq:layercake} gives the reverse implication with a sufficiently large small-ball exponent.
\end{proof}

\begin{remark}[The parameter and the induction hypothesis]\label{rem:fixedparameter}
A variable obtained after differentiation is generally not a standard Hermite random variable with the original self-similarity parameter $H$. It is a weighted multiple integral with the same exponent $\beta$. The condition $r\alpha<1$ holds at every intermediate level because $r\le q$ and $q\alpha<1$. The induction hypothesis is therefore formulated for these weighted kernels with fixed $\alpha$, not for the unweighted Hermite law at order $r$.
\end{remark}

\begin{remark}[Absence of a derivative comparison]
At level $r$, only scalar gradient estimates at level $r-1$ are assumed. They imply nondegeneracy of the vector of first derivatives at level $r$ through \eqref{eq:inductioncombination}. A density bound for that vector is then proved and used in \eqref{eq:finalbessel}. There is no circular use of level-$r$ nondegeneracy and no deterministic inequality that bounds a first derivative below by a higher derivative.
\end{remark}

\begin{remark}[Why compactness matters]\label{rem:compactness}
Uniform boundedness of the weights, even together with a uniform lower $L^2$ bound, is not a substitute for norm compactness. For example, take $T=2\pi$ and $w_n(s)=\sin(ns)$ at level one. These weights are bounded in $L^\infty$ and have constant $L^2$ norm. For every $x\in(0,T)$, the Riemann--Lebesgue lemma gives
\[
 \int_x^T\sin(ns)(s-x)^{-\beta}\dd s\longrightarrow0.
\]
The bound $(T-x)^{1-\beta}/(1-\beta)$ permits dominated convergence in $L^2(0,T)$. Hence $\norm{\Dt G_1(w_n)}_{\HT}\to0$, and uniform inverse-moment bounds fail. The compact families in the proof are finite-dimensional continuous images or explicitly assumed compact sets, not merely bounded collections of weights.
\end{remark}

\section{Weighted vectors and Hermite distributions}\label{sec:vectors}

\begin{proof}[Proof of Theorem~\ref{thm:mainvector}]
Consider
\begin{equation}\label{eq:weightsphere}
 \cK_X=\left\{\sum_{i=1}^n u_iw_i:u\in\Sphere^{n-1}\right\}.
\end{equation}
This is a compact subset of $L^\infty(0,T)$, and linear independence implies that it does not contain zero. By Theorem~\ref{thm:mainuniform}, for every $N>0$,
\[
 \sup_{|u|=1}\Prob\{\norm{\Dt(u^\top X)}_{\HT}\le\varepsilon\}
 \le C_N\varepsilon^N.
\]
The coordinates of $X$ have all Sobolev norms by Lemma~\ref{lem:kernels}. Lemma~\ref{lem:net} therefore gives all inverse moments of $\lambda_{\min}(\Gamma_X^T)$. Since
\begin{equation}\label{eq:detcomparison}
 \det\Gamma_X
 \ge\lambda_{\min}(\Gamma_X)^n
 \ge\lambda_{\min}(\Gamma_X^T)^n,
\end{equation}
all inverse determinant moments follow. Lemma~\ref{lem:density} gives a Schwartz density. If the weights are linearly dependent, some nonzero $u\in\R^n$ has $\sum_i u_iw_i=0$ almost everywhere. Linearity of \eqref{eq:kernel} gives $u^\top X=0$ almost surely, proving the converse.
\end{proof}

\begin{proposition}[Compact families of weighted vectors]\label{prop:uniformvectors}
Let $\mathcal V\subset(L^\infty(0,T))^n$ be compact in the product norm topology, and suppose that every tuple $(w_1,\ldots,w_n)\in\mathcal V$ is linearly independent. For fixed $1\le r\le q$, the corresponding vectors have uniformly bounded inverse moments of their restricted smallest Malliavin eigenvalues and of their full Malliavin determinants. Their densities form a bounded subset of $\Sch(\R^n)$.
\end{proposition}
\begin{proof}
The family
\[
 \left\{\sum_{i=1}^n u_iw_i:
       (w_1,\ldots,w_n)\in\mathcal V,\ |u|=1\right\}
\]
is compact in $L^\infty(0,T)$ and does not contain zero. Apply Theorem~\ref{thm:mainuniform}, then Lemmas~\ref{lem:net} and~\ref{lem:density}, retaining uniformity in the tuple. The determinant estimate follows from \eqref{eq:detcomparison}.
\end{proof}

\begin{proof}[Proof of Theorem~\ref{thm:hermite}]
The parameters \eqref{eq:hermiteparameters} satisfy \eqref{eq:parameters} for every $q\ge1$, since $q\alpha=2(1-H)<1$; the results of Sections~\ref{sec:deterministic}--\ref{sec:induction} therefore apply with $r=q$. For $T=1$ and $w=1$, Theorem~\ref{thm:mainuniform} and \eqref{eq:indicatorrepresentation} prove \eqref{eq:hermiteonetime}.

For finite-dimensional distributions, reorder the times so that $0<t_1<\cdots<t_n$ and set $T=t_n$. The weights $w_i=\ind_{(0,t_i)}$ are linearly independent. Indeed, if $\sum_i u_iw_i=0$ almost everywhere, examination of $(t_{n-1},t_n)$ gives $u_n=0$, examination of $(t_{n-2},t_{n-1})$ gives $u_{n-1}=0$, and so on, with $(0,t_1)$ giving the final coefficient. The case $n=1$ is immediate. Theorem~\ref{thm:mainvector} applies, and multiplying each component by $c_{H,q}>0$ preserves the inverse-moment and Schwartz conclusions.

For the stated increments take $T=b_n$ and $w_i=\ind_{(a_i,b_i)}$. These weights have disjoint supports of positive measure and are linearly independent. The same theorem applies directly. Thus neither independence of increments nor a determinant-factorization transfer is needed.
\end{proof}

\begin{remark}[Overlapping increments]
The weighted-vector theorem also covers overlapping increments whenever their indicator weights are linearly independent. Non-overlap is a simple sufficient condition, not a requirement of the argument. Conversely, any deterministic linear relation among the interval indicators yields the same almost sure relation among the increments.
\end{remark}

\begin{remark}[Scope of the uniformity]\label{rem:scopeuniformity}
The uniformity established here is over admissible compact families of weights at fixed order $q$, fixed kernel parameter $\alpha$ and fixed $T$. It is not asserted as $q$ tends to infinity, as $H$ approaches an endpoint of $(1/2,1)$, or as observation times coalesce. Different endpoint locations generally do not give a norm-continuous family of indicator weights in $L^\infty$: two distinct interval indicators have $L^\infty$ distance one. Consequently, Proposition~\ref{prop:uniformvectors} is not a grid-uniform theorem for arbitrary time grids, and it does not imply the grid-uniform density-tail estimates of~\cite{Moldavskaya2026}. Uniform bounds under moving endpoints or coalescing times require an additional argument and are not inferred here from the compact-family result.
\end{remark}

\begin{remark}[Dependence on the kernel structure]\label{rem:kernelstructure}
The conclusion does not hold for arbitrary elements of a fixed Wiener chaos. If $F=W(e)^2-1$ with $\norm e_{\Hh}=1$, then $\norm{DF}_{\Hh}=2|W(e)|$ and $\E\norm{DF}_{\Hh}^{-p}=\infty$ for $p\ge1$. The example shows why a general polynomial smoothness criterion must include a substantive nondegeneracy hypothesis; it does not rule out applying such a criterion once its hypotheses have been verified. The decisive property in the present setting is that multiplication by $a_{h_u}$ preserves nonvanishing throughout an arbitrarily large sphere of deterministic directions, and Lemma~\ref{lem:tail} supplies exactly this property for the Hermite kernel, at every chaos level and uniformly over admissible families.
\end{remark}

\appendix
\section{The fourth-order argument through quadratic forms}\label{app:four}

This appendix gives an alternative proof of the one-time result at order four. It is not used in the all-order induction. Its purpose is to exhibit the intermediate mechanism between the quadratic-form method at order three~\cite{Moldavskaya2026} and the uniform induction above.

Let $T=1$, $q=4$, $\alpha=(1-H)/2\in(0,1/4)$, and $F=Z_1^{H,4}$; we write $D_1=\Dt$ for $T=1$. Fix $0<\tau<1$. Repeated use of Lemma~\ref{lem:directional} gives, for $h,g\in C_c^\infty(0,\tau)$,
\begin{equation}\label{eq:fourderivatives}
 D_hF=4c_{H,4}G_3(a_h),\qquad
 D_gD_hF=12I_2(k_{h,g}),
\end{equation}
where
\begin{equation}\label{eq:fourkernel}
 k_{h,g}(x,y)
 =c_{H,4}\int_0^1a_h(s)a_g(s)(s-x)_+^{-\beta}(s-y)_+^{-\beta}\dd s
 =c_{H,4}K_2[a_ha_g](x,y).
\end{equation}
This kernel is symmetric and Hilbert--Schmidt by Lemma~\ref{lem:kernels}.

Define $U:L^2(0,1)\to L^2(\R)$ initially on smooth functions by
\[
 (Uu)(x)=\int_0^1u(s)(s-x)_+^{-\beta}\dd s.
\]
The energy identity and Schur's bound for $|s-t|^{-\alpha}$ give a bounded extension to $L^2(0,1)$. For $0<x<1$, its restriction is the right-sided fractional integral $\Gamma(1-\beta)I_{1-}^{1-\beta}u$. The semigroup argument used in the base case of Proposition~\ref{prop:induction}, valid also for $u\in L^2(0,1)\subset L^1(0,1)$, proves injectivity. Therefore
\begin{equation}\label{eq:fourrange}
 \overline{\Ran(U^*)}=L^2(0,1).
\end{equation}

Let $A_{h,g}$ be the operator associated with $k_{h,g}$. Pairing against smooth compactly supported test functions and applying absolute Fubini gives
\begin{equation}\label{eq:fourfactorization}
 A_{h,g}=c_{H,4}U M_{a_ha_g}U^*,
\end{equation}
where $M_{a_ha_g}$ is multiplication by $a_ha_g$. The identity extends by boundedness to all test vectors in $L^2(\R)$.

If $h,g$ are both nonzero, Lemma~\ref{lem:tail} gives $a_ha_g\ne0$ almost everywhere on $(\tau,1)$. A multiplication operator by such a function has infinite rank. Indeed, on a positive-measure set its absolute value is bounded below by a positive number, and indicators of infinitely many disjoint positive-measure subsets have linearly independent images. If $A_{h,g}$ had finite rank, injectivity of $U$ would imply that $M_{a_ha_g}\Ran(U^*)$ were finite dimensional and hence closed. By \eqref{eq:fourrange} and continuity of the multiplier, its closure contains $M_{a_ha_g}L^2(0,1)$, contradicting infinite rank. Thus
\begin{equation}\label{eq:fourinfiniterank}
 \rank(A_{h,g})=\infty\qquad(h\ne0,\ g\ne0).
\end{equation}

Choose real orthonormal families $h_1,\ldots,h_m$ and $g_1,\ldots,g_\ell$ in $C_c^\infty(0,\tau)$. Put $h_u=\sum_i u_ih_i$ and $g_v=\sum_jv_jg_j$ for unit vectors $u,v$. The map $(u,v)\mapsto A_{h_u,g_v}$ is a finite bilinear combination of fixed Hilbert--Schmidt operators and is therefore continuous in Hilbert--Schmidt norm. Since $|s_N(A)-s_N(B)|\le\norm{A-B}_{\mathrm{op}}\le\norm{A-B}_{\mathrm{HS}}$ for compact operators (Weyl's inequality for singular values), each $s_N(A_{h_u,g_v})$ is continuous on the compact set $\Sphere^{m-1}\times\Sphere^{\ell-1}$. For every fixed integer $N\ge1$, this continuity and \eqref{eq:fourinfiniterank} imply
\begin{equation}\label{eq:foursingularvalue}
 \kappa_N:=\min_{|u|=|v|=1}s_N(A_{h_u,g_v})>0.
\end{equation}

For $R_j(u)=D_{g_j}D_{h_u}F$ and $t=\rho v\in\R^\ell$,
\[
 \sum_{j=1}^\ell t_jR_j(u)=12\rho I_2(k_{h_u,g_v}).
\]
If $\lambda_j(u,v)$ are the eigenvalues of $A_{h_u,g_v}$, ordered by decreasing absolute value, the second-chaos expansion yields
\[
 I_2(k_{h_u,g_v})\ \overset{\mathrm{law}}=\
 \sum_{j\ge1}\lambda_j(u,v)(\xi_j^2-1),
\]
where the $\xi_j$ are independent standard normal variables and the series converges in $L^2$. Indeed, for $k=k_{h_u,g_v}$ the spectral theorem gives $k=\sum_j\lambda_j e_j\otimes e_j$ in $\Hh^{\otimes2}$ for an orthonormal family $(e_j)$. Since $I_2(e_j\otimes e_j)=W(e_j)^2-1$ and the variables $W(e_j)$ are independent standard normals, the $L^2$ isometry of $I_2$ yields the stated expansion. Taking characteristic-function moduli and using \eqref{eq:foursingularvalue} gives
\begin{equation}\label{eq:fourFourier}
 \begin{split}
 \left|\E\exp\left(i\sum_{j=1}^{\ell}t_jR_j(u)\right)\right|
 &=\prod_{j\ge1}\bigl(1+576\rho^2\lambda_j(u,v)^2\bigr)^{-1/4}\\
 &\le\bigl(1+576\kappa_N^2|t|^2\bigr)^{-N/4}.
 \end{split}
\end{equation}
The product is justified by $L^2$ convergence of the series, and each factor is the modulus $|\E\exp(i\vartheta(\xi^2-1))|=(1+4\vartheta^2)^{-1/4}$ with $\vartheta=12\rho\lambda_j(u,v)$. The inequality keeps only the factors with $j\le N$ and uses $|\lambda_j(u,v)|=s_j(A_{h_u,g_v})\ge\kappa_N$ for these $j$, since $A_{h_u,g_v}$ is self-adjoint. If $N>2\ell$, the last bound is integrable over $\R^\ell$, uniformly in $u$. The vectors $(R_1(u),\ldots,R_\ell(u))$ therefore have uniformly bounded densities.

Bessel's inequality now implies
\[
 \sup_{|u|=1}\Prob\{\norm{D_1(D_{h_u}F)}_{L^2(0,1)}\le\varepsilon\}
 \le C_{m,\ell}\varepsilon^\ell.
\]
Since $\ell$ is arbitrary, Lemma~\ref{lem:net} applies to the vector $(D_{h_i}F)_{i=1}^m$. Its coordinates have uniformly finite Sobolev norms for this fixed family, and Lemma~\ref{lem:density} gives a bounded joint density. A final use of Bessel's inequality yields
\[
 \Prob\{\norm{D_1F}_{L^2(0,1)}\le\varepsilon\}
 \le C_m\varepsilon^m\qquad\text{for every }m\ge1.
\]
Equation~\eqref{eq:layercake} proves all inverse moments. Higher derivatives are used to establish a density for the vector of first derivatives, not to compare the two derivative norms pointwise.

% Keep the references and author address together.

\end{document}